\documentclass[12pt]{article}
\usepackage[T1]{fontenc}
\usepackage{verbatim}
\usepackage[margin=2.5cm]{geometry}
\usepackage{algorithm2e}
\usepackage{here}

\usepackage{cite}

\usepackage{setspace}
\usepackage{algpseudocode}
\usepackage{dsfont}
\usepackage{hyperref}
\usepackage{url}
\usepackage{mathtools}
\usepackage[inline]{enumitem}
\usepackage{microtype}
\usepackage{graphicx}
\usepackage{nicematrix}
\usepackage{amsmath}
\usepackage{amssymb}
\usepackage{amsthm}
\usepackage{mathtools}
\usepackage{comment}
\usepackage{nicefrac}
\usepackage{amsfonts,latexsym,graphics}
\usepackage{color}
\usepackage{multirow}
\usepackage[english]{babel}
\usepackage{epsfig}
\usepackage{bm}
\usepackage{makeidx}
\usepackage{subcaption}
\usepackage{adjustbox}
\usepackage{booktabs}
\usepackage{tabularray}

\usepackage{tikz}
\usetikzlibrary{tikzmark}
\usetikzlibrary{positioning}
\usetikzlibrary{graphs}
\usetikzlibrary{graphs.standard}
\usetikzlibrary{decorations.pathmorphing}
\usetikzlibrary{decorations.text}
\pgfdeclarelayer{bg}
\pgfsetlayers{bg,main}

\usepackage{authblk}

\usepackage{xspace}

\newcommand{\norm}[2]{\| #1 \|_{#2}}

\def\matrix0{{\mbox {\boldmath $O$}}}

\def\vecs{{\mbox{\boldmath $s$}}}
\def\vect{{\mbox{\boldmath $t$}}}

\def\vecx{{\mbox{\boldmath $x$}}}
\def\vecy{{\mbox{\boldmath $y$}}}

\def\j{{\mbox{\boldmath $1$}}}
\def\vec0{\mbox{\bf 0}}

\def\vecnu{{\mbox{\boldmath $\nu$}}}
\def\vecchi{{\mbox{\boldmath $\chi$}}}

\theoremstyle{definition}
\newtheorem{definition}{Definition}
\newtheorem{example}[definition]{Example}
\newtheorem{observation}[definition]{Observation}
\newtheorem{remark}[definition]{Remark}

\theoremstyle{plain}
\newtheorem{proposition}[definition]{Proposition}
\newtheorem{lemma}[definition]{Lemma}
\newtheorem{theorem}[definition]{Theorem}

\newtheorem{corollary}[definition]{Corollary}
\newtheorem{conjecture}[definition]{Conjecture}

\newcommand*{\claimproofname}{Proof}

\newcommand\tran{\mkern-2mu\raise1.25ex\hbox{$\scriptscriptstyle\top$}\mkern-3.5mu}

\def\dist{\mathop{\rm dist }\nolimits}

\def\diam{\mathrm{diam}}

\DeclarePairedDelimiter\abs{\lvert}{\rvert}
\DeclarePairedDelimiter\inn{\langle}{\rangle}
\DeclarePairedDelimiter\floor{\lfloor}{\rfloor}
\DeclarePairedDelimiter\ceil{\lceil}{\rceil}

\DeclareMathOperator{\Sym}{Sym}

\SetArgSty{}
\SetKwFor{For}{for}{do}{}
\SetKwInOut{Input}{Input}
\SetKwInOut{Output}{Output}

\makeatletter
\let\c@figure\c@table
\let\ftype@figure\ftype@table
\let\ext@figure\ext@table

\makeatother 

\title{The Zero Forcing Number of Graph Powers}

\author[1,2]{Aida Abiad}
\author[3]{Mary Flagg}
\author[4,5]{Sina Ghasemi Nezhad}
\author[1]{Sjanne Zeijlemaker}

\affil[1]{Department of Mathematics and Computer Science, Eindhoven University of Technology, The Netherlands}
\affil[2]{Department of Mathematics and Data Science, Vrije Universiteit Brussel, Belgium}
\affil[3]{Department of Mathematics, Statistics and Computer Science, University of St.~Thomas, Houston, TX, USA}
\affil[4]{Department of Computer Science, Vrije Universiteit Amsterdam, The Netherlands}
\affil[5]{Informatics Institute, Faculty of Science, University of Amsterdam, The Netherlands}

\affil[ ]{
    \texttt{a.abiad.monge@tue.nl},
    \texttt{flaggm@stthom.edu},
    \texttt{s.ghaseminezhad@student.vu.nl}
}

\date{}

\begin{document}

\maketitle

\begin{abstract}
The $k$-th power of a simple graph $G$, denoted $G^k$, is the graph with vertex set $V(G)$ where two vertices are adjacent if they are within distance $k$ in $G$. We investigate the zero forcing number of graph powers. Powers of graphs are much denser and generally not encompassed by existing results on zero forcing of graphs, hence their study requires a different approach. In contrast with the usual zero forcing behavior under edge deletion, we show that the zero forcing parameter (and variations of it) is monotone with respect to taking powers. We also determine the zero forcing number of powers of paths and cycles, together with upper bounds for powers of spiders and of grids. We then present spectral lower bounds on the zero forcing number of graph powers which uniquely use the spectrum of the base graph, as well as linear programming methods to compute these bounds. Finally, we study the sharpness of the derived bounds. To derive our results we use techniques ranging from graph theory, linear algebra and polynomial optimization. \\


\end{abstract}

\section{Introduction}\label{sec:introduction}

For a positive integer $k$, the $k$-th \emph{graph power} of a graph $G=(V, E)$, denoted by $G^k$, is the graph with vertex set $V$ in which two distinct vertices are joined by an edge if there is a path in $G$ of length at most $k$ between them. Graph powers help design efficient algorithms for specific combinatorial optimization problems, see e.g. \cite{AB1995,BC2000}. In distributed computing, the $k$-th power of graph $G$ represents the possible flow of information during $k$ rounds of communication in a distributed network of processors organized according to a graph $G$ \cite{l92}. Thus, the investigation of the behavior of various parameters of powers of a fixed graph leads to many fascinating problems, some of which are motivated by questions in information theory and coding theory. There are also some results of combinatorial nature which have focused on graph powers: the chromatic number \cite{AM2002,H2009,Hav2009,KK1969,KP2016}, the Shannon capacity \cite{AL2007v1,AL2007v2}, the rainbow connection number \cite{Basavaraju2014RainbowProducts}, or the edge growth of graph squares \cite{G2014}, among others. Less is known regarding the spectral properties of the graph power in terms of the spectral properties of the base graph  \cite{acfnz21,Das2013LaplacianGraph,JZ2017}. Motivated by this, in this paper we focus on investigating the zero forcing number of graph powers, with a main focus on deriving spectral bounds which only use the eigenvalues of the base graph $G$. 

Zero forcing is a propagation process on a graph where the vertices are initially partitioned into two sets of black and white vertices. A white vertex is colored black (\emph{forced}) if it is the unique white neighbor of a black vertex. The minimum number of initial black vertices needed to force all vertices of a graph $G$ is called the \emph{zero forcing number}. The zero forcing number was formally introduced in \cite{AIM} as an upper bound for the minimum rank, a connection which was also established in \cite{A}. The study of the zero forcing number is also motivated by the investigation of quantum networks, influence in social networks and power dominating sets, see e.g. \cite{burgarth2009local} and~\cite{dean2011power}. The positive semidefinite zero forcing number was introduced in \cite{barioli2010} as a variant of the zero forcing number. The maximum positive semidefinite nullity problem (equivalently, the minimum positive semidefinite rank problem) for a graph $G$ is to determine the largest possible nullity (smallest possible rank) over all real positive semidefinite matrices described by the graph. The motivation of the positive semidefinite zero forcing is to bound the maximum positive semidefinite nullity.

This paper introduces the first systematic study of the zero forcing number of the $k$-th graph power $Z(G^k)$ (and some variants of zero forcing such as positive semidefinite zero forcing). While the zero forcing number is generally non-monotone under edge addition or deletion \cite{EHHLR2012}, we show that it behaves monotonically with respect to taking graph powers. Some results on graph powers are known from the matrix side: the minimum rank of powers of paths was determined in~\cite{AGKSMY2012}, and powers of cycles are consecutive circulants, whose minimum rank was determined in~\cite{DM2016}. In this direction, we determine the exact zero forcing numbers for powers of paths and cycles, and establish upper bounds for powers of spiders and grids. Furthermore, we derive new spectral lower bounds for $Z(G^k)$ using polynomials of the base graph's adjacency matrix, model the optimal bound polynomial choice via linear programming techniques, and evaluate the bounds sharpness through computational experiments.

\section{Preliminaries}\label{sec:preliminaries}

A graph $G=(V(G),E(G))$ consists of a set of vertices $V(G)$ and a set of edges $E(G)$, which are two-element subsets of the vertex set. The edge $\{v,w\}$ is denoted $vw$. Two vertices $v$ and $w$ are \emph{adjacent} or \emph{neighbors} if $vw \in E(G)$. We denote this by $v\sim w$. The \emph{neighborhood} of a vertex is $N(v)=\{w:vw \in E(G)\}$. The \emph{closed neighborhood} of $v$ is $N(v) \cup \{v\}$. 

Given a positive integer $t$, the neighborhood of $v$ in $G^t$ will be denoted $N^t(v)$. The \emph{distance} $\dist_G(v,w)$ between two vertices $v,w\in V(G)$ is the length of a shortest $v$--$w$ path in $G$, with $\dist_G(v,w)=\infty$ if no such path exists, and $\diam(G) = \max_{v,w\in V(G)}\dist_G(v,w)$ denotes the \emph{diameter} of $G$. Thus $v\sim w$ in $G^k$ if and only if $1\leq \dist_G(v,w)\leq k$. We write $\delta(G)$ for the minimum degree of $G$, $P_n$ for the path on $n$ vertices, $C_n$ for the cycle on $n$ vertices, and $G\mathbin{\square} H$ for the Cartesian product of $G$ and $H$.

As mentioned, zero forcing is a propagation process on a graph where the vertices are initially partitioned into two sets of black and white vertices. A white vertex is colored black (\emph{forced}) if it is the unique white neighbor of a black vertex. The minimum number of initial black vertices needed to force all vertices of a graph $G$ is called the \emph{zero forcing number}. 

The \emph{positive semidefinite color change rule} is as follows: let $G$ be a graph and $B$ be the set consisting of all the black vertices of $G$. Let $W_1,\dots, W_p$ be the set of vertices of the $p$ components of $G\setminus B$ with $p\geq 1$. If $u\in B$ and $w$ is the only white neighbor of $u$ in the graph induced by $V (W_i \cup B)$, then change the color of $w$ to black. A \emph{positive zero forcing set} $S_p$ of $G$, is a subset of vertices of $G$ such that if initially the vertices in $S_p$ are colored black and the remaining vertices are colored white, then the derived set of $G$ is $V(G)$. The \emph{positive zero forcing number} of $G$, denoted by $Z_{+}(G)$, is then defined to be the smallest size of a positive zero forcing set of $G$.

Throughout, a \emph{chronological list of forces} of a zero forcing set $B$ is the ordered list of the forces performed, in the order in which they were performed, and the \emph{derived set} of $B$ is the set of black vertices obtained once no further force is possible. We will repeatedly use the following elementary observation.

\begin{observation}\label{obs:atmostonce}
    In any chronological list of forces, every vertex performs at most one force. Indeed, immediately after $u$ forces its unique white neighbor, $u$ has no white neighbor left, and vertices never turn white again.
\end{observation}

In the positive semidefinite setting, a \emph{single force} is one application of the positive semidefinite color change rule into one fixed component of $G-B$. A black vertex may force into several components at once, but the individual forces can always be linearly ordered, since the vertex forces into each component independently of the others.

We also fix some algebraic notation. We write $\mathbb{R}_k[x]$ for the set of real polynomials of degree at most $k$, $\vecchi_S\in\{0,1\}^{V}$ for the characteristic vector of a subset $S\subseteq V$, and $\j$ for the all-one vector. For a connected graph $G$ on $n$ vertices with adjacency matrix $A$, the \emph{Perron eigenvector} $\vecnu$ is the entry-wise positive eigenvector of $A$ associated with the largest eigenvalue $\lambda_1$, normalized so that $\norm{\vecnu}{2}=1$; if $G$ is $d$-regular, then $\lambda_1=d$ and $\vecnu = \j/\sqrt{n}$. For an eigenvalue $\mu$ of $A$, we denote by $\operatorname{mult}_A(\mu)$ its multiplicity.

Finally, $\mathcal{S}(G)$ denotes the set of real symmetric $n\times n$ matrices $X=(x_{uv})$ whose off-diagonal support is described by $G$, that is, such that for $u\neq v$ we have $x_{uv}\neq 0$ if and only if $uv\in E(G)$; the diagonal entries of $X$ are unconstrained. The \emph{maximum nullity} of $G$ is $M(G) = \max\{\operatorname{null}(X): X\in \mathcal{S}(G)\}$ and the \emph{minimum rank} of $G$ is $\operatorname{mr}(G) = \min\{\operatorname{rank}(X):X\in\mathcal{S}(G)\} = \abs{V(G)} - M(G)$. The zero forcing number bounds the maximum nullity from above:
\begin{equation}\label{eq:M-leq-Z}
    M(G) \leq Z(G) \qquad \text{for every graph } G,
\end{equation}
see \cite{AIM}. We will also use the elementary lower bound $Z(G)\geq \delta(G)$ from \cite{AIM}.

\section{Monotonicity and Exact Values}\label{sec:monotonicity}

While it is known that the zero forcing number is not a monotone parameter when adding or removing an edge \cite{EHHLR2012}, the following results show that when considering graph powers, the zero forcing number behaves nicely. In fact, the result that zero forcing number is monotone increasing with respect to the power of the graph is true in a more general setting. 

Amos et al. proposed \emph{$k$-forcing} in \cite{ACDP2015} as the following generalization of the zero forcing game. Note that we will refer to this as $b$-forcing to avoid confusion with the power of the graph. Assume the vertices of a graph are colored in two colors, say white and black. Iteratively apply the following color change rule: if $u$ is a black vertex with at most $b$ white neighbors, then change the color of all the neighbors of $u$ to black. When no more color changes are possible, if all vertices are black, the original set of black vertices is a $b$-forcing set of $G$. The original zero forcing is $1$-forcing under this definition. The minimum cardinality of a $b$-forcing set of the graph $G$ is the $b$-forcing number of $G$, denoted $Z_b(G)$.

\begin{theorem}\label{thm:monotonicity-k-force}
    Let $G$ be a graph and $b \geq 1$. Then for $k \geq 1$, $Z_b(G^k) \leq Z_b(G^{k+1})$.  
\end{theorem}

\begin{proof}
    Let $B$ be a $b$-forcing set for $G^{k+1}$. Let $\mathcal{F}_{k+1}$ be a chronological list of forces for $B$ on $G^{k+1}$. Note that each force consists of a single vertex forcing a set of vertices of size at most $b$. We will show the result by induction. As a basis for induction, let $v_1 \rightarrow W_1=\{w_1, w_2, \dots, w_t\}$ be the first force for $B$ in $G^{k+1}$. Let $W_1'$ be $W_1 \cap N^{k}(v_1)$. Since the only white neighbors of $v_1$ in $G^k$ are the vertices in $W_1'$, $v_1$ forces $W_1'$ in $G^k$. If $x \in W_1\setminus W_1'$, then $\dist_G(x,v_1)=k+1$ and there exists $u_1 \in N(v_1)$ such that $x \in N^k(u_1)$. Since $N^k(u_1) \subseteq N^{k+1}(v_1)$, $u_1$ can force its white neighbors black. Continue to choose neighbors of $v_1$ to force any additional white vertices in $W_1$ until all of $W_1$ is forced black. 

    Inductively proceed by following the order of the forces in $\mathcal{F_{k+1}}$. Assuming that the first $j-1$ forces were performed, at step $j$, proceed to force $W_j$ according to the procedure of the first force. At the end of the process all vertices of $G^k$ are black, and therefore $B$ is a $b$-forcing set for $G^k$.
\end{proof}

\begin{corollary}\label{cor:monotone-Z}
    Let $G$ be a graph and let $k \geq 1$. Then $Z(G^k) \leq Z(G^{k+1})$.
\end{corollary}

Similar to standard zero forcing and $k$-forcing, monotonicity also holds for positive semidefinite zero forcing. 

\begin{theorem}\label{thm:monotonicity-PSD}
    Let $G$ be a graph and $k \geq 1$. Then $Z_{+}(G^{k}) \leq Z_{+}(G^{k+1})$.
\end{theorem}

\begin{proof}
    Suppose $B$ is a set of currently colored vertices.  For $k \geq 1$, let $W^k_1, W^k_2,\dots,W^k_t$ be the connected components of $G^{k}-B$ and define $W^{k}(x)$ to be the component of $G^{k}-B$ containing the vertex $x$. If $w$ and $x$ are in the same connected component of $G^{k}-B$, then they are also in the same connected component of $G^{k+1}-B$ since all of the edges in $G^{k}$ are also present in $G^{k+1}$. Therefore, for $w \notin B$, $W^{k}(w) \subseteq W^{k+1}(w)$.

    Let $B_0$ be an initial set of black vertices that is a PSD zero forcing set of $G^{k+1}$, and let $B_i$ be the set of black vertices after the $i$-th forcing step. Let $\mathcal{F}$ be a chronological list of single forces for $B_0$ on $G^{k+1}$. Note that a vertex $v \in B_i$ may force into more than one connected component of $G^k-B_i$, but there is no loss of generality in ordering the individual forces since the $v$ forces into each component independently of the other components. The goal is to inductively show that $B_0$ is a PSD zero forcing set for $G^k$.

    As a basis for induction, consider the first force in $\mathcal{F}$. For some $v \in B_0$ and $w \notin B_0$, $v$ forces $w$ black in $G^{k+1}$. This implies $w$ is the only white neighbor of $v$ in $B_0 \cup W^{k+1}(w)$. Since $W^{k}(w) \subseteq W^{k+1}(w)$, the only possible white neighbor of $v$ in $W^{k}(w)$ is $w$. If $w$ is a neighbor of $v$ in $G^{k}$, then $v \rightarrow w$ is a valid force in $G^{k}$. If $w \notin N^k(v)$, then $\dist_G(v,w)=k+1$. Let $u$ be a neighbor of $v$ in $G$ on a path of length $k+1$ from $v$ to $w$ in $G$. Then $u \in B_0$: indeed, if $u$ were white, then $\dist_G(u,w) = k$ would give $u\sim w$ in $G^{k+1}$ and hence $u \in W^{k+1}(w)$, so that $u$ would be a white neighbor of $v$ in $B_0\cup W^{k+1}(w)$ distinct from $w$, contradicting the validity of the force $v\rightarrow w$ in $G^{k+1}$. Moreover $N^k(u) \subseteq N^{k+1}(v)$ and $w$ is the only white neighbor of $u$ in $B_0 \cup W^{k}(w)$. Therefore $u$ forces $w$ black.

    Assume that at forcing step $m$, all of the vertices forced in the first $m-1$ steps have been forced in $G^{k}$. Then at step $m$, $v_m \rightarrow w_m$ in $G^{k+1}$ which implies $w_m$ is the only white neighbor of $v_m$ in $W^{k+1}(w_m)$. If $\dist_G(v_m,w_m) \leq k$, then $w_m$ is the only white neighbor of $v_m$ in $G^{k}$, and $v_m$ forces $w_m$ to turn black. If $w_m$ is not a neighbor of $v_m$ in $G^{k}$, then there exists $u_m \in N(v_m)$ on a shortest path from $v_m$ to $w_m$ in $G$. As in the base case, $u_m$ is black and $w_m$ is its only white neighbor in $W^{k}(w_m)$. Thus $u_m$ forces $w_m$ black. Since $B_0$ is a PSD zero forcing set for $G^{k+1}$, all vertices are black at the end of the forcing process described by $\mathcal{F}$ in $G^{k+1}$. By induction, all vertices in $G^{k}$ are also forced black using $B_0$. Therefore, $B_0$ is also a PSD zero forcing set of $G^{k}$.
\end{proof}

Regarding consecutive graph powers, the automated conjecturing program TxGraffiti \cite{Dprivate} produced the following inequality as a candidate conjecture:
\begin{equation}\label{conj:TxGraffiti}
    Z(G^{t+1}) \leq 2Z(G^t)-1.
\end{equation}
However, the following simple graph example shows that this conjecture is not true.

\begin{figure}[!h]
    \centering
    \scalebox{0.5}{
        \begin{tikzpicture}[scale=2,every node/.style={draw,shape=circle,outer sep=2pt,inner sep=1pt,minimum size=.2cm}]		
            \node[fill=none]  (1) at (-4,1) {};
            \node[fill=none]  (2) at (-2,1) {};
            \node[fill=none]  (3) at (0,1) {};
            \node[fill=none]  (4) at (2,1) {};
            \node[fill=none]  (5) at (4,1) {};
            \node[fill=none]  (6) at (4,-1) {};
            \node[fill=none]  (7) at (2,-1) {};
            \node[fill=none]  (8) at (0,-1) {};
            \node[fill=none]  (9) at (-2,-1) {};
            \node[fill=none]  (10) at (-4,-1) {};
            
            \draw (1)--(2)--(3)--(4)--(5)--(6)--(7)--(8)--(9)--(10)--(1);
            \draw (1)--(9)--(3)--(7)--(5);
            \draw (2)--(9);
            \draw (3)--(8);
            \draw (4)--(7);
        \end{tikzpicture}
    }
    \caption{A counterexample to the TxGraffiti conjecture \eqref{conj:TxGraffiti}.}\label{fig:counterexample-TxGraffiti}
\end{figure}
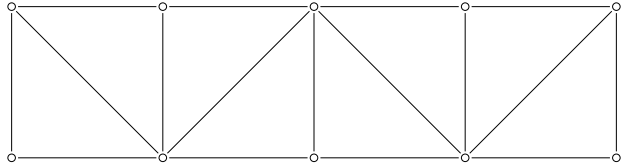

Indeed, let $G$ be the graph in Figure \ref{fig:counterexample-TxGraffiti}, that is, $G$ is a graph on two parallel paths, and $Z(G)=2$. However, $G^2$ has minimum degree $\delta(G^2)=5$ and since $Z(H)\geq \delta(H)$ for every graph $H$ \cite{AIM}, we obtain $Z(G^2) \geq 5 > 3 = 2Z(G)-1$. This example generalizes to two parallel paths of arbitrary length. 

\begin{remark}
    Let $P_5$ be the path on the five vertices $v_1,\dots,v_5$, where $v_i\sim v_{i+1}$ for $i=1,2,3,4$. The second power $P^2_5$ has zero forcing number two, as $\{v_1, v_2\}$ is a zero forcing set. However, the graph $P^2_5 - \{v_2,v_4\}$ has zero forcing number three. This means that the zero forcing number of graph powers is not monotone under edge removals. 
\end{remark}

By Corollary~\ref{cor:monotone-Z} the sequence $Z(G^k)$ is non-decreasing in $k$. The gap between two consecutive terms is, however, not bounded by any constant.

\begin{proposition}\label{prop:gap-unbounded}
    For every $m \geq 1$ and $k\geq 1$ there is a graph $G$ with $Z(G^{k+1})-Z(G^{k}) = m$. Indeed, if $G$ is the disjoint union of $m$ paths $P_\ell$ with $\ell \geq k+3$, then $Z(G^k) = mk$.
\end{proposition}

\begin{proof}
    Zero forcing is additive over connected components and $(G\cup H)^k = G^k\cup H^k$, so it suffices to prove $Z(P_\ell^k) = k$ for $\ell \geq k+2$. The minimum degree of $P_\ell^k$ is $k$, so $Z(P_\ell^k) \geq k$. Conversely the first $k$ vertices force: with $v_1,\dots,v_k$ black, $v_1$ has $v_{k+1}$ as its unique white neighbor, and inductively $v_j$ forces $v_{j+k}$ for $j = 1,2,\dots,\ell-k$.
\end{proof}

We now determine, or bound, the zero forcing number of the powers of some natural graph families. For paths and cycles the answer can be read off from the literature on the minimum rank of graph powers.

\begin{proposition}\label{prop:paths}
    Let $\ell \geq k+2$. Then $Z(P_\ell^k) = M(P_\ell^k) = k$.
\end{proposition}

\begin{proof}
    The equality $Z(P_\ell^k)=k$ was shown in the proof of Proposition~\ref{prop:gap-unbounded}. For the maximum nullity, $\operatorname{mr}(P_\ell^k) = \ell-k$ by \cite{AGKSMY2012}, whence $M(P_\ell^k) = \ell-\operatorname{mr}(P_\ell^k) = k$. Alternatively, $\mathcal{S}(P_\ell^k)$ consists precisely of the symmetric banded matrices of bandwidth $k$ whose outermost diagonal has no zero entry, and any such matrix has rank at least $\ell-k$, with equality attained.
\end{proof}

\begin{proposition}\label{prop:cycles}
    Let $n \geq 2k+1$. Then $Z(C_n^k) = M(C_n^k) = 2k$.
\end{proposition}

\begin{proof}
    The graph $C_n^k$ is $2k$-regular, so $Z(C_n^k) \geq \delta(C_n^k) = 2k$ by \cite{AIM}. Conversely, label $V(C_n) = \{v_1,\dots,v_n\}$ cyclically and color $B = \{v_1,\dots,v_{2k}\}$ black. The neighbors of $v_{k+1+j}$ in $C_n^k$ are the vertices $v_{1+j},\dots,v_{2k+1+j}$ other than $v_{k+1+j}$ itself, so if $v_1,\dots,v_{2k+j}$ are black then $v_{2k+1+j}$ is the unique white neighbor of $v_{k+1+j}$. Performing these forces for $j = 0,1,\dots,n-2k-1$ colors the whole graph black, so $Z(C_n^k)\le 2k$. Finally, $C_n^k$ is the consecutive circulant $C_n(1,2,\dots,k)$, and $M(C_n^k) = 2k$ by \cite{DM2016}.
\end{proof}

Beyond paths and cycles the situation is less clear, and we content ourselves with upper bounds. For $m\geq 2$ and $\ell \geq 1$, let $S_m(\ell)$ denote the \emph{spider} obtained from $m$ paths on $\ell$ vertices by joining one endpoint of each of them to a new central vertex $c$; we write $v_{i,1},\dots,v_{i,\ell}$ for the vertices of the $i$-th leg, indexed by their distance to $c$. Note that $S_2(\ell) = P_{2\ell+1}$.

\begin{proposition}\label{prop:spider}
    Let $m \geq 2$, $k \geq 1$ and $\ell \geq k+1$. Then $Z\bigl(S_m(\ell)^k\bigr) \leq (m-1)k$.
\end{proposition}

\begin{proof}
    Distances in $T = S_m(\ell)$ are given by $\dist_T(v_{i,j},v_{i,j'}) = \abs{j-j'}$, by $\dist_T(v_{i,j},v_{i',j'}) = j+j'$ for $i\neq i'$, and by $\dist_T(c,v_{i,j}) = j$. Consequently, in $T^k$ the vertex $v_{i,j}$ is adjacent to $v_{i,j'}$ exactly when $0<\abs{j-j'}\leq k$, to $v_{i',j'}$ with $i'\neq i$ exactly when $j+j'\leq k$, and to $c$ exactly when $j\leq k$.

    Let
    $$
    	B = \{v_{i,j}\ :\ 1\leq i\leq m-1,\ \ell-k+1\leq j\leq \ell\}
    $$
    consist of the $k$ outermost vertices of each of the first $m-1$ legs, so that $\abs{B} = (m-1)k$. We exhibit a chronological list of forces for $B$ in $T^k$.

    Process the legs $i=1,2,\dots,m-1$ in turn and, within leg $i$, process $t = \ell-k,\ell-k-1,\dots,1$ in this order, letting $v_{i,t+k}$ force $v_{i,t}$. When leg $i$ is processed at step $t$, the black set contains $B$, all of the legs $1,\dots,i-1$, and $v_{i,t+1},\dots,v_{i,\ell}$. Since $t\geq 1$, we have $t+k>k$, so $v_{i,t+k}$ is adjacent neither to $c$ nor to any vertex of another leg. Its neighbors inside leg $i$ are the $v_{i,j}$ with $t\leq j\leq \min\{\ell,t+2k\}$ and $j\neq t+k$, and all of them except $v_{i,t}$ are already black. Hence the force is valid, and after Phase~1 the legs $1,\dots,m-1$ are entirely black.

   The neighbors of $v_{1,k}$ are $c$, the vertices $v_{1,j}$ with $1\leq j\leq 2k$ and $j\neq k$, and the vertices $v_{i',j'}$ with $i'\neq 1$ and $j'\leq 0$, of which there are none. All of them except $c$ are black, so $v_{1,k}\rightarrow c$.

    Assume $k\geq 2$ and let $j$ run through $1,2,\dots,k-1$ in this order. At step $j$, every vertex is black except $v_{m,j},v_{m,j+1},\dots,v_{m,\ell}$. The neighbors of $v_{1,k-j}$ are the center $c$, some vertices of leg $1$, and the vertices $v_{i',j'}$ with $i'\neq 1$ and $j'\leq k-(k-j) = j$. Among the latter, those lying on the legs $2,\dots,m-1$ are black after Phase~1, and those lying on leg $m$ are $v_{m,1},\dots,v_{m,j}$, of which only $v_{m,j}$ is white. Hence $v_{1,k-j}\rightarrow v_{m,j}$.

    At this point every vertex is black except $v_{m,k},v_{m,k+1},\dots,v_{m,\ell}$. The neighbors of $c$ in $T^k$ are exactly the vertices $v_{i,j}$ with $j\leq k$, and among these only $v_{m,k}$ is white; hence $c\rightarrow v_{m,k}$.

    Finally, for $t = k+1,k+2,\dots,\ell$, the vertex $v_{m,t-k}$ forces $v_{m,t}$: its neighbors are the $v_{m,j}$ with $\abs{j-(t-k)}\leq k$, together with $c$ and, if $t-k\leq k$, some vertices of the other legs, and all of these except $v_{m,t}$ are black.

    Therefore $B$ is a zero forcing set of $T^k$ of size $(m-1)k$.
\end{proof}

For $k=1$ the bound of Proposition~\ref{prop:spider} is attained, since $Z(T) = P(T) = m-1$ for every spider $T = S_m(\ell)$ with $m\geq 2$, where $P$ denotes the path cover number \cite{AIM}. An exhaustive computation for $2\leq m\leq 5$, $\ell\in\{4,5\}$ and $1\leq k\leq 3$ gives equality in every case considered, which suggests the following.

\begin{conjecture}\label{conj:spider}
    Let $m \geq 2$, $k \geq 1$ and $\ell \geq k+1$. Then $Z\bigl(S_m(\ell)^k\bigr) = (m-1)k$.
\end{conjecture}

A similar construction works for grids.

\begin{proposition}\label{prop:grid}
    Let $b \geq 1$, $k \geq 1$ and $\ell \geq k+1$. Then $Z\bigl((P_b\mathbin{\square} P_\ell)^k\bigr) \leq bk$.
\end{proposition}

\begin{proof}
    Write the vertices of $P_b\mathbin{\square} P_\ell$ as pairs $(r,t)$ with $1\leq r\leq b$ and $1\leq t\leq \ell$, so that the distance between $(r,t)$ and $(r',t')$ equals $\abs{r-r'}+\abs{t-t'}$. Color black the first $k$ columns, that is, let $B = \{(r,t)\ :\ 1\leq r\leq b,\ 1\leq t\leq k\}$, so that $\abs{B} = bk$. Suppose that all of the columns $1,\dots,t$ are black for some $t\geq k$ and let $1\leq r\leq b$. The neighbors of $(r,t-k+1)$ in the $k$-th power are the vertices $(r',t')$ with $\abs{r-r'}+\abs{t-k+1-t'}\leq k$; such a vertex lies outside the first $t$ columns only if $t'\geq t+1$, which forces $\abs{t-k+1-t'}\geq k$ and therefore $r'=r$ and $t'=t+1$. Hence $(r,t+1)$ is the unique white neighbor of $(r,t-k+1)$ and can be forced. Applying this for $r=1,\dots,b$ turns column $t+1$ black, and the claim follows by induction on $t$.
\end{proof}

An exhaustive computation for $1\leq b\leq 3$, $1\leq k\leq 3$ and $k+1\leq \ell\leq 10$ shows that the bound of Proposition~\ref{prop:grid} is attained in every case considered; for $k=1$ this is the classical equality $Z(P_b\mathbin{\square} P_\ell) = \min\{b,\ell\}$ \cite{AIM}.

\begin{conjecture}\label{conj:grid}
    Let $b \geq 1$, $k \geq 1$ and $\ell \geq k+1$. Then $Z\bigl((P_b\mathbin{\square} P_\ell)^k\bigr) = bk$.
\end{conjecture}

\section{Spectral Bounds}\label{sec:spectral-bounds}

In this section we prove several spectral bounds for the zero forcing (and positive semidefinite zero forcing) number of a power graph $G^k$ in terms of the eigenvalues and multiplicities of the original graph $G$. Not many spectral bounds for $Z(G)$ are known in the literature \cite{kalinowski2019zero,ZWWJ2022,AZ2024}, thus our work in this section contributes to extend this line of research. The technical device used throughout is to apply a polynomial $p$ of degree at most $k$ to the adjacency matrix of $G$, so that the support of $p(A)$ is contained in that of $G^k$; this idea goes back to the eigenvalue bounds for the $k$-independence number obtained in \cite{acf19,acfnz21}.

\subsection{Expander Mixing Lemma}

The following result is one of the few spectral lower bounds known on the zero forcing number. In this section we show that this result holds in the more general case of graph powers. To do so, we will use some preliminary spectral results and linear optimization techniques.

\begin{theorem}{\cite[Theorem 1.4(i)]{kalinowski2019zero}}\label{thm:Suakovetal-Zevbound}
    Let $G$ be a $d$-regular graph on $n$ vertices with adjacency eigenvalues $\lambda_1\geq \dots \geq \lambda_n$ and let $\lambda = \max\{\abs{\lambda_2},\abs{\lambda_n}\}$. Then
    $$
        Z(G) \geq n\left(1-\frac{2\lambda}{d+\lambda}\right).
    $$
\end{theorem}

Theorem~\ref{thm:Suakovetal-Zevbound} only holds for regular graphs; it was extended to general graphs in \cite{AZ2024}. Although the spectra of $G$ and $G^k$ are in general not related, we can nevertheless bound $Z(G^k)$ in terms of the spectrum of $G$. The key tool is the following polynomial version of the expander mixing lemma, which generalizes~\cite[Theorem 18(ii)]{AZ2024}.

\begin{theorem}\label{thm:general-poly-EML}
    Let $G=(V,E)$ be a connected graph on $n$ vertices with eigenvalues $\lambda_1 \geq \dots \geq \lambda_n$, adjacency matrix $A$ and normalized Perron eigenvector $\vecnu$. Let $p\in\mathbb{R}_k[x]$ and let $\hat{\lambda}\in \{\lambda_2,\dots,\lambda_n\}$ such that $\abs{p(\hat{\lambda})}\geq \abs{p(\lambda_i)}$ for $i\geq 2$. For $S, T\subseteq V$, define $P(S,T)\coloneqq\sum_{i\in S,j\in T}p(A)_{ij}$. If $p(A)$ has nonnegative entries and $p(\lambda_1) > |p(\hat{\lambda})|$, then for any nonempty $S,T\subseteq V(G)$, we have
    \begin{align*}
        \left|P(S,T) - p(\lambda_1)\langle\vecchi_S,\vecnu\rangle \langle\vecchi_T,\vecnu\rangle\right| &\leq |p(\hat{\lambda})|\sqrt{(\abs{S}-\langle\vecchi_S,\vecnu\rangle^2)(|T|-\langle\vecchi_T,\vecnu\rangle^2)}\\
        &< |p(\hat{\lambda})|\sqrt{|S||T|}.
    \end{align*}
\end{theorem}

\begin{proof}
    This follows directly from~\cite[Theorem 9]{AZ2024} taking $\vecx=\vecy=\vecnu$, $B=p(A)$ and $R=C=p(\lambda_1)I$.
\end{proof}

Let $\vecs$ and $\vect$ be $m$-tuples of vertices such that $s_i\neq s_j$, and $t_i\neq t_j$ for all $i\neq j$. We call $(\vecs,\vect)$ an \emph{$m$-witness} if $s_i\sim t_i$ for all $i\in [m]$ and, for all $i<j$, both $s_i\not\sim t_j$ and $s_i\neq t_j$. Note that no condition is imposed on the pairs $(s_i,t_j)$ with $i>j$; in particular, the tuples $\vecs$ and $\vect$ may have entries in common.

\begin{lemma}\label{lem:grundy-LB}
    Let $G$ be a graph on $n$ vertices. Then $Z(G) \geq n - m$, where $m$ is the size of a largest witness in $G$.
\end{lemma}

\begin{proof}
    This is essentially \cite[Corollary 2.3]{BRESAR2017}, see also \cite[Lemma 3.1]{kalinowski2019zero}; we include a short argument for completeness. Let $B$ be a zero forcing set of $G$ with $\abs{B} = Z(G)$ and let $u_1\rightarrow w_1,\ u_2\rightarrow w_2,\ \dots,\ u_m\rightarrow w_m$ be a chronological list of forces for $B$. Every vertex outside $B$ is forced exactly once, so $m = n-Z(G)$ and the vertices $w_1,\dots,w_m$ are pairwise distinct; by Observation~\ref{obs:atmostonce}, the vertices $u_1,\dots,u_m$ are pairwise distinct as well. Put $\vecs = (u_1,\dots,u_m)$ and $\vect = (w_1,\dots,w_m)$. By construction $s_i\sim t_i$ for every $i$. Let $i<j$. At the moment when $u_i$ performs its force, the vertex $w_j$ is still white and $u_i$ is black, so $w_j\neq u_i$; moreover $w_j$ is not a neighbor of $u_i$, since otherwise $u_i$ would have had the two distinct white neighbors $w_i$ and $w_j$. Hence $(\vecs,\vect)$ is an $m$-witness of size $n-Z(G)$, which is therefore at most the size of a largest witness.
\end{proof}

Now we are ready to state the main result and spectral bounds of this paper.

\begin{theorem}\label{thm:zero-forcing-power}
    Let $G$ be a $d$-regular graph on $n$ vertices with $\diam(G) > k$, adjacency matrix $A$ and eigenvalues $\lambda_1 \geq \dots \geq \lambda_n$. Then for any polynomial $p\in \mathbb{R}_k[x]$ such that $p(A)$ has nonnegative entries and $p(d) > \abs{p(\hat{\lambda})} > 0$,
    $$
        Z(G^k) \geq \floor*{n\left(1-\frac{2\abs{p(\hat{\lambda})}}{p(d)+\abs{p(\hat{\lambda})}} \right)}.
    $$
\end{theorem}

\begin{proof}
    Let $(\vecs,\vect)$ be a largest $m$-witness in $G^k$ and set $q \coloneqq \abs{p(\hat{\lambda})} / (p(d) + \abs{p(\hat{\lambda})})\in (0,1)$. We will prove the case of even and odd $m$ separately.
    
    \begin{description}
        \item[Case $m$ even.]
            Let $m = 2\mu n$ and put $S=\{s_1,s_2,\dots,s_{\mu n}\}$ and $T=\{t_{\mu n+1},t_{\mu n+2},\dots,t_{m}\}$. Every index occurring in $S$ is smaller than every index occurring in $T$, so by the definition of a witness there are no edges of $G^k$ between $S$ and $T$ and moreover $S\cap T = \emptyset$. Consequently $P(S,T) = 0$: indeed, $\deg p\leq k$ confines the support of $p(A)$ to the pairs of vertices at distance at most $k$ in $G$, that is, to the pairs that are equal or adjacent in $G^k$, and neither occurs for $u\in S$ and $v\in T$. Since $G$ is $d$-regular, we also know that $\vecnu = \j / \sqrt{n}$, where $\j$ denotes the all-one vector. It follows from Theorem~\ref{thm:general-poly-EML} that 
            \begin{align}
                0 &= \sum_{u\in S,v\in T}p(A)_{uv}\nonumber \\
                &\geq p(d)\frac{(\mu n)^2}{n} - \abs{p(\hat{\lambda})}\left(\mu n - \frac{(\mu n)^2}{n}\right)\label{eq:zero-condition}\\
                &= \mu n(p(d)\mu - \abs{p(\hat{\lambda})}(1-\mu)).\nonumber 
            \end{align}
            and hence~$\mu \leq q$. Note that if $\abs{p(\hat{\lambda})} = 0$, then also $p(d)=0$ by Equation~\eqref{eq:zero-condition}. To prevent division by zero, we therefore require $\abs{p(\hat{\lambda})} > 0$. Using Lemma~\ref{lem:grundy-LB}, we obtain
            $$
                Z(G^k) \geq n - m = n(1-2\mu) \geq n(1-2q) \geq \floor*{n(1-2q)}.
            $$
        \item[Case $m$ odd.]
            Let $m = 2\mu n + 1$ and put $S=\{s_1,s_2,\dots,s_{\mu n}\}$ and $T=\{t_{\mu n+1},t_{\mu n+2},\dots,t_{m}\}$. Note that if $G^k$ is complete, a largest witness has size one and $S$ is the empty set. We therefore require $\diam(G) > k$. As in the previous case, $S\cap T = \emptyset$, $P(S,T)=0$ and $\vecnu = \j / \sqrt{n}$. If $\mu n = 0$, then $m = 1$ and $Z(G^k)\geq n-1 \geq \floor*{n(1-2q)}$, because $q>0$. Assume therefore that $\mu n \geq 1$. It follows from Theorem~\ref{thm:general-poly-EML} that 
            \begin{align*}
                0 &= \sum_{u\in S,v\in T}p(A)_{uv}\\
                &\geq p(d)\frac{\mu n(\mu n+1)}{n} - \abs{p(\hat{\lambda})}\sqrt{\mu n(\mu n + 1)\left(1 - \frac{\mu n}{n}\right)\left(1 - \frac{\mu n+1}{n}\right)}\\
                &\geq p(d)\frac{\mu n(\mu n+1)}{n} - \abs{p(\hat{\lambda})}\left(1 - \frac{\mu n}{n}\right)\sqrt{\mu n(\mu n + 1)}\\
                &\geq \sqrt{\mu n(\mu n + 1)} \left( p(d)\frac{\sqrt{\mu n(\mu n + 1)}}{n} - \abs{p(\hat{\lambda})}\left(1 - \mu\right)\right), 
            \end{align*}
            from which we obtain
            $$
                0 \geq p(d)\frac{\sqrt{\mu n(\mu n + 1)}}{n} - \abs{p(\hat{\lambda})}\left(1 - \mu\right).
            $$
            Since $\mu n\geq 1$, we have $\sqrt{\mu n(\mu n+1)} > \mu n$, and therefore $p(d)\mu - \abs{p(\hat{\lambda})}(1-\mu) < 0$, that is, $\mu < q$ strictly. Hence $m = 2\mu n+1 < 2qn+1$, and since $m$ is an integer, $m\leq \ceil*{2qn}$. Using Lemma~\ref{lem:grundy-LB} once more,
            \[
                Z(G^k) \geq n - m \geq n - \ceil*{2qn} = \floor*{n-2qn} = \floor*{n(1-2q)}.\qedhere
            \]
    \end{description}
\end{proof}

\begin{remark}\label{rem:floor}
    The integer part in Theorem~\ref{thm:zero-forcing-power} originates in the case of a witness of odd size. If a largest witness of $G^k$ has even size, then the proof gives the unrounded bound $Z(G^k)\geq n\bigl(1-2\abs{p(\hat{\lambda})}/(p(d)+\abs{p(\hat{\lambda})})\bigr)$. If it has odd size, then the sets $S$ and $T$ have the different cardinalities $\mu n$ and $\mu n+1$, and the strict inequality $\sqrt{\mu n(\mu n+1)}>\mu n$ compensates for this asymmetry up to the integer part. The same rounding is implicit in Theorem~\ref{thm:Suakovetal-Zevbound}, whose proof only considers witnesses of even size. In any case the loss is smaller than one, and it vanishes whenever the right-hand side happens to be an integer.
\end{remark}

Note that if we choose $k=1$ and $p(x)=x$ in Theorem~\ref{thm:zero-forcing-power}, we recover Theorem~\ref{thm:Suakovetal-Zevbound}, up to the integer part, as a corollary.

Theorem~\ref{thm:zero-forcing-power} can be extended to non-regular graphs, at the cost of a factor depending on the Perron eigenvector.

\begin{theorem}\label{thm:zero-forcing-power-irregular}
    Let $G$ be a connected graph on $n$ vertices with $\diam(G) > k$, adjacency matrix $A$, eigenvalues $\lambda_1 \geq \dots \geq \lambda_n$ and normalized Perron eigenvector $\vecnu$, and let $\nu_{\min} = \min_i \nu_i > 0$. Then for any $p \in \mathbb{R}_k[x]$ such that $p(A)$ has nonnegative entries and $p(\lambda_1) > \abs{p(\hat\lambda)} > 0$,
	$$
	   Z(G^k) \geq \floor*{\,n-\frac{2\abs{p(\hat{\lambda})}}{\nu_{\min}^2\left(p(\lambda_1)+\abs{p(\hat{\lambda})}\right)}}.
	$$
\end{theorem}

\begin{proof}
    Let $(\vecs,\vect)$ be a largest $m$-witness in $G^k$ and write
    $$
        Q \coloneqq \frac{2\abs{p(\hat\lambda)}}{\nu_{\min}^2\bigl(p(\lambda_1)+\abs{p(\hat\lambda)}\bigr)} > 0,
    $$
    so that the claimed bound reads $Z(G^k)\geq \floor*{n-Q}$.

    Suppose first that $m = 2\mu$ is even and put $S = \{s_1,\dots,s_{\mu}\}$ and $T = \{t_{\mu+1},\dots,t_{m}\}$. As in the proof of Theorem~\ref{thm:zero-forcing-power}, no edge of $G^k$ joins $S$ and $T$ and $S\cap T = \emptyset$, so $P(S,T) = 0$, since $\deg p \leq k$ confines the support of $p(A)$ to pairs of vertices at distance at most $k$. By Theorem~\ref{thm:general-poly-EML},
    $$
        0 \geq p(\lambda_1)\inn{\vecchi_S,\vecnu}\inn{\vecchi_T,\vecnu}-\abs{p(\hat\lambda)}\sqrt{\bigl(\mu-\inn{\vecchi_S,\vecnu}^2\bigr)\bigl(\mu-\inn{\vecchi_T,\vecnu}^2\bigr)}.
    $$
    Since $\vecnu > 0$, we have $\inn{\vecchi_S,\vecnu} \geq \mu\nu_{\min}$ and $\mu-\inn{\vecchi_S,\vecnu}^2 \leq \mu(1-\mu\nu_{\min}^2)$, and likewise for $T$. Substituting,
    $$
        p(\lambda_1)\mu^2\nu_{\min}^2 \leq \abs{p(\hat\lambda)}\mu(1-\mu\nu_{\min}^2),
    $$
    whence $\mu\nu_{\min}^2\bigl(p(\lambda_1)+\abs{p(\hat\lambda)}\bigr) \leq \abs{p(\hat\lambda)}$, that is, $m = 2\mu\leq Q$. By Lemma~\ref{lem:grundy-LB} we conclude that $Z(G^k)\geq n-m\geq n-Q\geq \floor*{n-Q}$.

    Suppose now that $m = 2\mu+1$ is odd, and take $S = \{s_1,\dots,s_{\mu}\}$ and $T = \{t_{\mu+1},\dots,t_{m}\}$, so that $\abs{S} = \mu$ and $\abs{T} = \mu+1$. If $\mu = 0$, then $m=1$ and $Z(G^k)\geq n-1\geq \floor*{n-Q}$, because $Q>0$. Assume therefore that $\mu\geq 1$. Arguing as above, $\inn{\vecchi_S,\vecnu}\inn{\vecchi_T,\vecnu}\geq \mu(\mu+1)\nu_{\min}^2$ and
    \begin{align*}
        \sqrt{\bigl(\mu-\inn{\vecchi_S,\vecnu}^2\bigr)\bigl(\mu+1-\inn{\vecchi_T,\vecnu}^2\bigr)}
        &\leq \sqrt{\mu(\mu+1)}\sqrt{(1-\mu\nu_{\min}^2)\bigl(1-(\mu+1)\nu_{\min}^2\bigr)}\\
        &\leq \sqrt{\mu(\mu+1)}\,(1-\mu\nu_{\min}^2),
    \end{align*}
    so that $p(\lambda_1)\mu(\mu+1)\nu_{\min}^2\leq \abs{p(\hat\lambda)}\sqrt{\mu(\mu+1)}\,(1-\mu\nu_{\min}^2)$ and hence
    $$
        p(\lambda_1)\sqrt{\mu(\mu+1)}\,\nu_{\min}^2\leq \abs{p(\hat\lambda)}\bigl(1-\mu\nu_{\min}^2\bigr).
    $$
    Since $\mu\geq 1$ we have $\sqrt{\mu(\mu+1)}>\mu$, so that $p(\lambda_1)\mu\nu_{\min}^2 < \abs{p(\hat\lambda)}(1-\mu\nu_{\min}^2)$, that is, $\mu\nu_{\min}^2\bigl(p(\lambda_1)+\abs{p(\hat\lambda)}\bigr) < \abs{p(\hat\lambda)}$ and $m = 2\mu+1 < Q+1$. As $m$ is an integer, $m\leq \ceil*{Q}$, and Lemma~\ref{lem:grundy-LB} gives $Z(G^k)\geq n-m\geq n-\ceil*{Q} = \floor*{n-Q}$.
\end{proof}

If $G$ is $d$-regular, then $\vecnu = \j/\sqrt{n}$ and $\nu_{\min}^2 = 1/n$, so Theorem~\ref{thm:zero-forcing-power-irregular} specializes exactly to Theorem~\ref{thm:zero-forcing-power}. On the other hand, the bound degrades quickly when $\nu_{\min}$ is small, that is, for graphs that are far from regular. Replacing the crude estimate $\inn{\vecchi_S,\vecnu}\geq \abs{S}\nu_{\min}$ by one based on weight-equitable partitions, in the spirit of \cite{AZ2024}, is a natural direction for further work.

\subsubsection{Optimization of the Bound in Theorem~\ref{thm:zero-forcing-power}}

Note that the lower bound in Theorem~\ref{thm:zero-forcing-power} is valid for many different choices of a polynomial $p$. A natural question is whether we can determine the polynomial $p$ which gives the best possible bound. In this section, we formulate a mixed-integer linear program (MILP) which finds a polynomial that maximizes the lower bound in Theorem~\ref{thm:zero-forcing-power}. Note that while MILP is known to be NP-hard in general, in order to model the bound from Theorem \ref{thm:zero-forcing-power} we will propose an optimization program which is actually in P since it only has one (binary) integer variable. 

The proposed MILP model has variables $a_0,a_1\dots,a_k\in \mathbb{R}$ to represent the coefficients of $p$, i.e., $p(x) = a_0 + a_1x + \dots + a_k x^k$. Since it is not clear beforehand for which $i=2,\dots,n$ we have $\hat{\lambda} = \lambda_i$, we will run the MILP once for each distinct eigenvalue $\lambda_{\ell}$ with the assumption that $\hat{\lambda} = \lambda_{\ell}$. 

Note that the bound in Theorem~\ref{thm:zero-forcing-power} is invariant under multiplying $p$ by a positive constant. We will therefore assume without loss of generality that $p(d) + \abs{p(\hat{\lambda})} = 1$. In particular, this means that $\abs{p(\hat{\lambda})} < 1$. We will use this fact to derive a linear representation of the term $\abs{p(\hat{\lambda})}$. 

Let $y^+,y^-\in \mathbb{R}_{\geq 0}$ be nonnegative variables. We would like to have $p(\hat{\lambda}) = y^+$ and $y^- = 0$ if $p(\hat{\lambda}) \geq 0$ and $p(\hat{\lambda}) = -y^-$ and $y^+ = 0$ otherwise. We can ensure this by setting $p(\hat{\lambda}) = y^+-y^-$ and introducing a binary variable $b\in \{0,1\}$ with constraints $y^+ \leq b$, $y^- \leq 1-b$. Since $\abs{p(\hat{\lambda})} < 1$, these constraints ensure that $y^+ = p(\hat{\lambda})$ if $p(\hat{\lambda}) \geq 0$ and $-y^- = p(\hat{\lambda})$ otherwise. Using these additional variables, we now have $\abs{p(\hat{\lambda})} = y^+ + y^-$, which is a linear expression.

With the additional variables $y^+,y^-$ and the assumption $p(d) + \abs{p(\hat{\lambda})} = 1$, we can linearize the objective function as
$$
    n\left(1 - \frac{2\abs{p(\hat{\lambda})}}{p(d) + \abs{p(\hat{\lambda})}} \right) = n\left(1 - 2(y^+ + y^-) \right).
$$
MILP~\eqref{eq:MILP} shows how the remaining conditions of Theorem~\ref{thm:zero-forcing-power} can be converted to linear constraints. 

\begin{equation}
    \def\arraystretch{1.3}
    \boxed{
    \begin{tabular}{r l l l l}
        {\tt maximize} & $n\left( 1 - 2(y^++y^-) \right)$ & & \\
        {\tt s.t.} & $0 \leq y^+ \leq b$ & & \multirow{3}{*}[-0.5em]{$\begin{rcases*}\ \\ \ \\ \ \end{rcases*}$ } & \multirow{3}{*}[-0.5em]{linearizing $\abs{p(\hat{\lambda})}$} \\
        & $0 \leq y^- \leq 1-b$ &&&\\
        & $y^+ - y^- = \displaystyle \sum_{i=0}^k a_i\hat{\lambda}^i$ & & & \\
        & $\displaystyle \sum_{i=0}^k a_i d^i > y^+ + y^-$ & & $\begin{rcases*} \ \end{rcases*}$ & $p(d) > |p(\hat{\lambda})|$\\
        & $\displaystyle \sum_{i=0}^k a_i d^i + y^+ + y^- = 1$ & & $\begin{rcases*} \ \end{rcases*}$ &  $p(d) + \abs{p(\hat{\lambda})} = 1$\\
        & $y^+ + y^- \geq \displaystyle
              \sum_{i=0}^k a_i \lambda_j^i,$ & $j = 2,\dots, n$ &
        \multirow{2}{*}[-0.8em]{$\begin{rcases*} \ \\ \ \end{rcases*}$} & \multirow{2}{*}[-0.7em]{$\abs{p(\hat{\lambda})} \geq \abs{p(\lambda_j)}$} \\
        & $y^+ + y^- \geq \displaystyle -\sum_{i=0}^k a_i \lambda_j^i,$ & $j = 2,\dots, n$ & & \\
        & $\displaystyle \sum_{i=0}^k a_i (A^i)_{uv} \geq 0$ & $u,v\in V$ & $\begin{rcases*} \ \end{rcases*}$ & $p(A)$ nonnegative \\
        & $y^+ + y^- > 0$ & & $\begin{rcases*} \ \end{rcases*}$ & $\abs{p(\hat{\lambda})} > 0$ \\
        & $y^+, y^-, a_0,\dots, a_k\in \mathbb{R}$ &  & & \\
        & $b\in \{0,1\}$ & & & 
    \end{tabular}
    }\label{eq:MILP}
    \def\arraystretch{1}
\end{equation}

Note that, while our formulation is technically a MILP, it only contains one binary variable $b$, so it is effectively equivalent to solving to two linear programs, one for each possible value of $b$. We can therefore compute its optimal value efficiently, as we illustrate in the next section.

\subsubsection{Bounds Sharpness}

In the Table~\ref{tab:MILP-results} below we compare the bound from MILP~\eqref{eq:MILP} to the zero forcing number for $k=2,3$. We limit ourselves to regular graphs with at most 100 vertices, as we need to compute the exact zero forcing number for comparison, which is NP-hard. However, the MILP can be applied to much larger graphs. The bound is tight for several graphs, which is indicated in bold.

\begin{table}[!htb]
\tiny
    \centering
    \begin{tabular}{|l||l | l||l | l|}
        \hline
         & \multicolumn{2}{c||}{$k=2$} & \multicolumn{2}{c|}{$k=3$} \\\hline\hline
        Name & Theorem~\ref{thm:zero-forcing-power} & $Z(G^2)$ & Theorem~\ref{thm:zero-forcing-power} & $Z(G^3)$\\
        \hline
        Balaban 10-cage & $20$ & $34$ & $37$ & $49$ \\
        Bidiakis cube & $7$ & $8$ & - & - \\
        Blanu\v{s}a first snark graph & $7$ & $11$ & $11$ & $15$ \\
        Blanu\v{s}a second snark graph & $7$ & $10$ & $11$ & $14$ \\
        Brinkmann graph & $16$ & $17$ & - & - \\
        Bucky ball & $8$ & $20$ & $16$ & $30$ \\
        Coclique graph of Hoffmann-Singleton graph & $72$ & $86$ & $95$ & $97$ \\
        Conway-Smith graph for $3\cdot \Sym(7)$ & $49$ & $56$ & $60$ & $61$ \\
        Coxeter graph & $16$ & $18$ & $23$ & $24$ \\
        Desargues graph & $10$ & $14$ & $15$ & $16$ \\
        Dodecahedron & $9$ & $12$ & $\textbf{16}$ & $\textbf{16}$ \\
        Double star snark & $8$ & $18$ & $15$ & $25$ \\
        Durer graph & $6$ & $8$ & $9$ & $10$ \\
        Dyck graph & $13$ & $18$ & $22$ & $25$ \\
        Ellingham-Horton 54-graph & $4$ & $22$ & $6$ & $30$ \\
        Ellingham-Horton 78-graph & $2$ & $29$ & $3$ & $40$ \\
        F26A graph & $13$ & $16$ & $19$ & $21$ \\
        Flower snark & $10$ & $13$ & $15$ & $17$ \\
        Folkman graph & $11$ & $16$ & $\textbf{17}$ & $\textbf{17}$ \\
        Foster graph & $25$ & $42$ & $48$ & $63$ \\
        Foster graph for $3\cdot\Sym(6).2$ graph & $35$ & $39$ & $42$ & $43$ \\
        Franklin graph & $7$ & $8$ & - & - \\
        Frucht graph & $6$ & $7$ & $9$ & $10$ \\
        Gosset graph & $\bf{54}$ & $\bf{54}$ & - & - \\
        Gray graph & $15$ & $29$ & $29$ & $39$ \\
        Harborth graph & $3$ & $22$ & $5$ & $30$ \\
        Harries graph & $20$ & $34$ & $37$ & $50$ \\
        Harries-Wong graph & $20$ & $34$ & $37$ & $49$ \\
        Heawood graph & $\bf{10}$ & $\bf{10}$ & - & - \\
        Hexahedron & $\textbf{6}$ & $\textbf{6}$ & - & - \\
        Hoffman graph & $11$ & $12$ & $\textbf{14}$ & $\textbf{14}$ \\
        Holt graph & $19$ & $21$ & - & - \\
        Horton graph & $3$ & $42$ & $6$ & $54$ \\
        Icosahedron & $\textbf{10}$ & $\textbf{10}$ & - & - \\
        Klein 3-regular graph & $19$ & $26$ & $34$ & $41$ \\
        Klein 7-regular graph & $21$ & $22$ & - & - \\
        Markstroem graph & $5$ & $11$ & $9$ & $15$ \\
        McGee graph & $13$ & $15$ & $20$ & $22$ \\
        Meredith graph & $5$ & $37$ & $9$ & $49$ \\
        Moebius-Kantor graph & $10$ & $11$ & $\textbf{14}$ & $\textbf{14}$ \\
        Nauru graph & $12$ & $15$ & $18$ & $20$ \\
        Pappus graph & $11$ & $12$ & $15$ & $16$ \\
        Perkel graph & $46$ & $50$ & - & - \\
        Robertson graph & $15$ & $16$ & - & - \\
        Sylvester graph & $29$ & $32$ & - & - \\
        Szekeres snark graph & $8$ & $23$ & $14$ & $29$ \\
        Tietze graph & $7$ & $9$ & - & - \\
        Truncated tetrahedron & $7$ & $8$ & - & - \\
        Tutte-Coxeter graph & $16$ & $19$ & $24$ & $25$ \\
        Tutte graph & $4$ & $18$ & $7$ & $26$ \\
        Watkins snark graph & $6$ & $19$ & $12$ & $24$ \\
        Wells graph & $27$ & $28$ & $\textbf{30}$ & $\textbf{30}$ \\
         \hline
    \end{tabular}
    \caption{Comparison of the bound in Theorem~\ref{thm:zero-forcing-power} to the exact values of $Z(G^2)$ and $Z(G^3)$ on named SageMath graphs with at most 100 vertices. A dash indicates the bound is not applicable for $k=3$, as the graph has diameter three. Bold entries indicate tightness of our bound.}\label{tab:MILP-results}
\end{table}

The bound of Theorem~\ref{thm:zero-forcing-power} is tight for several of the graphs in Table~\ref{tab:MILP-results}, but it can also be very far from the true value. The following result explains this behavior: the quality of the bound is governed by the isoperimetric structure of $G^k$, and no choice of polynomial can beat the resulting ceiling.

\begin{proposition}\label{prop:isoperimetric}
    Let $G$ be a connected $d$-regular graph with $\diam(G) > k$, and let $\mathrm{opt}$ denote the optimal value of MILP~\eqref{eq:MILP}. For every $S\subseteq V$ with $\abs{S} \leq n/2$, writing $e_k(S)$ for the number of edges of $G^k$ between $S$ and $V\setminus S$ and $\beta = 2e_k(S)/\abs{S}$, if $\beta < 2$ then
    $$
        \mathrm{opt} \leq \frac{n\beta}{2-\beta} \leq \frac{2n e_k(S)}{(2-\beta)^{-1}(2-\beta)\abs{S}} = \frac{2n e_k(S)}{\abs{S}}\cdot\frac{1}{2-\beta}.
    $$
    In particular $\mathrm{opt} \leq 2n e_k(S)/\abs{S}$ whenever $\beta \leq 1$.
\end{proposition}

\begin{proof}
    Let $p$ be feasible for MILP~\eqref{eq:MILP}, normalized as there so that $p(d)+\abs{p(\hat\lambda)} = 1$, and write $W = p(A)$, $y = \abs{p(\hat\lambda)}$. Then $W$ is symmetric with nonnegative entries, $W\j = p(d)\j$, and $\norm{W}{2} = \max_i\abs{p(\lambda_i)} = p(d)$, since the feasibility constraint gives $p(d) > \abs{p(\lambda_i)}$ for $i \geq 2$; in particular every entry of $W$ is at most $p(d)$ and $p(d)$ is a simple eigenvalue of $W$ with eigenvector $\j$. Let $x = \vecchi_S-\frac{\abs{S}}{n}\j\perp\j$. A standard computation using the constant row sums gives
    $$
        x^{\tran}Wx = p(d)\norm{x}{2}^2-\sum_{uv\in E(G^k), \abs{\{u,v\}\cap S} = 1} W_{uv} \geq p(d)\norm{x}{2}^2 - e_k(S) p(d),
    $$
    and $\norm{x}{2}^2 = \abs{S}(n-\abs{S})/n \geq \abs{S}/2$. Since the second largest eigenvalue of $W$ equals $\max_{i\ge2}p(\lambda_i)$ and dominates the Rayleigh quotient of $x$,
    $$
        y \geq \max_{i \geq 2}p(\lambda_i) \geq p(d)\Bigl(1-\frac{e_k(S)}{\norm{x}{2}^2}\Bigr) \geq p(d)(1-\beta) = (1-y)(1-\beta),
    $$
    so $y \geq (1-\beta)/(2-\beta)$ and $n(1-2y) \leq n\beta/(2-\beta)$.
\end{proof}

Proposition~\ref{prop:isoperimetric} explains the polarity visible in Table~\ref{tab:MILP-results}. Whenever $G$ admits a balanced vertex set with a small $G^k$-boundary, the bound collapses to a constant, regardless of $n$. The graphs with the worst rows of Table~\ref{tab:MILP-results} are exactly the long, cyclically thin cubic graphs with small balanced separators and small spectral gap, whereas the tight rows are strongly regular or distance-regular graphs with large gap.

Powers of cycles illustrate the phenomenon in its purest form: if $S$ is an arc of $\floor*{n/2}$ consecutive vertices of $C_n$, then $e_k(S) = k(k+1)$ and $\beta\to 0$ as $n\to\infty$, so that $\mathrm{opt}\leq 2k(k+1)(1+o(1))$ is bounded by a constant, while $Z(C_n^k) = 2k$ by Proposition~\ref{prop:cycles}. More generally, any lower bound obtained by applying Theorem~\ref{thm:general-poly-EML} to a bipartition induced by a witness is limited by the conductance of $G^k$; on poor expanders it should be complemented by bounds of a different nature, such as the minimum degree bound $Z(G^k)\geq \delta(G^k)$ or the multiplicity bound of the next subsection.

\subsection{Multiplicities}

Next we derive a lower bound on $Z(G^k)$ in terms of the eigenvalue multiplicities of the base graph. The mechanism is inequality~\eqref{eq:M-leq-Z}: it suffices to exhibit a single matrix in $\mathcal{S}(G^k)$ with large nullity, and a suitable polynomial in the adjacency matrix of $G$ does the job.

\begin{theorem}\label{thm:multiplicities}
    Let $G$ be a graph with adjacency matrix $A$, let $k \geq 1$, and let
    $\Lambda = \{\mu_1,\dots,\mu_k\}$ be a set of $k$ distinct eigenvalues of $A$. Set $p(x) = \prod_{i=1}^{k}(x-\mu_i)$ and suppose that
    \begin{equation}\label{eq:support-condition}
        p(A)_{uv} \neq 0 \quad \text{for every pair } u \neq v \text{ with } \dist_G(u,v) \leq k-1.
    \end{equation}
    Then $p(A) \in \mathcal{S}(G^k)$ and consequently
    $$
        Z(G^k) \geq M(G^k) \geq \sum_{i=1}^{k}\operatorname{mult}_A(\mu_i).
    $$
\end{theorem}

\begin{proof}
    Since $p$ is monic of degree $k$, we have $p(A) = A^k+\sum_{i=0}^{k-1}a_iA^i$ for some reals $a_i$. If $\dist_G(u,v) > k$, then $(A^i)_{uv} = 0$ for all $i \leq k$, so $p(A)_{uv} = 0$. If $\dist_G(u,v) = k$, then $(A^i)_{uv} = 0$ for $i < k$, while $(A^k)_{uv}$ counts the walks of length $k$ from $u$ to $v$, all of which are geodesics; hence $p(A)_{uv} = (A^k)_{uv} \geq 1$. Together with hypothesis~\eqref{eq:support-condition} this shows that, for $u \neq v$, $p(A)_{uv} \neq 0$ exactly when $\dist_G(u,v) \leq k$, that is, $p(A) \in \mathcal{S}(G^k)$. Since the eigenvalues of $p(A)$ are $p(\lambda)$ for $\lambda$ ranging over the spectrum of $A$, the nullity of $p(A)$ equals $\sum_{i=1}^{k}\operatorname{mult}_A(\mu_i)$, and the claim follows from $\operatorname{null}(p(A)) \leq M(G^k) \leq Z(G^k)$.
\end{proof}

\begin{remark}
    Theorem~\ref{thm:multiplicities} is really a family of bounds, one for each admissible root set $\Lambda$, and one is free to optimize over $\Lambda$. Only the multiplicities enter the conclusion, but the support condition~\eqref{eq:support-condition} does depend on which eigenvalues are chosen, so the set of the $k$ largest multiplicities need not be admissible. For $k=1$ the statement reduces to the familiar fact that $Z(G)\geq M(G)\geq \operatorname{mult}_A(\mu)$ for every eigenvalue $\mu$ of $A$, since then $p(A) = A-\mu I\in\mathcal{S}(G)$ and~\eqref{eq:support-condition} is vacuous. Finally, nothing in the proof uses that $A$ is a $0/1$ matrix: the same argument applies verbatim to any nonnegative matrix in $\mathcal{S}(G)$.
\end{remark}

The next two examples show that Theorem~\ref{thm:multiplicities} can be sharp, and also that the support condition~\eqref{eq:support-condition} is a genuine restriction.

\begin{example}\label{ex:petersen}
    Let $G$ be the Petersen graph, whose spectrum is $\{3^1,1^5,(-2)^4\}$ and whose adjacency matrix satisfies $A^2 = J+2I-A$, where $J$ denotes the all-one matrix. Since $\diam(G) = 2$, we have $G^2 = K_{10}$. Taking $k = 2$ and $\Lambda = \{1,-2\}$ gives $p(x) = (x-1)(x+2) = x^2+x-2$ and $p(A) = J$, which has no zero entry, so that~\eqref{eq:support-condition} holds and $p(A)\in\mathcal{S}(K_{10})$. Theorem~\ref{thm:multiplicities} yields $Z(G^2)\geq M(G^2)\geq 5+4 = 9$, and since $Z(K_{10}) = 9$ the bound is sharp: $Z(G^2) = M(G^2) = 9$.
\end{example}

\begin{example}\label{ex:spider}
    Let $T = S_3(2)$ be the spider with three legs of length two, that is, the tree on seven vertices obtained by subdividing every edge of the star $K_{1,3}$; note that $\ell = 2 < k+1$ for $k=2$, so Proposition~\ref{prop:spider} does not apply here. The adjacency spectrum of $T$ is $\{2,1^2,0,(-1)^2,-2\}$. The choice $k=2$ and $\Lambda = \{1,-1\}$ fails: here $p(x) = x^2-1$, and for adjacent $u,v$ the entry $p(A)_{uv} = (A^2)_{uv}$ counts the common neighbors of $u$ and $v$, which is $0$ because $T$ is triangle-free; since $\dist_T(u,v) = 1\leq k-1$, condition~\eqref{eq:support-condition} is violated. The choice $\Lambda = \{1,0\}$ does work: now $p(x) = x^2-x$ and $p(A)_{uv} = (A^2)_{uv}-A_{uv} = -1\neq 0$ for adjacent $u,v$, which is all that~\eqref{eq:support-condition} requires. Theorem~\ref{thm:multiplicities} therefore gives $Z(T^2)\geq M(T^2)\geq 2+1 = 3$. Conversely, writing $c$ for the center, $a_i$ for the neighbors of $c$ and $b_i$ for the leaves, the set $\{c,b_1,b_2\}$ is a zero forcing set of $T^2$: the vertex $b_i$ forces $a_i$ for $i = 1,2$, then $a_1$ forces $a_3$ and $a_3$ forces $b_3$. Hence $Z(T^2) = M(T^2) = 3$ and the bound is again sharp.
\end{example}

\section*{Acknowledgments} 

Aida Abiad is supported by NWO (Dutch Research Council) through the grant VI.Vidi.213.085. Mary Flagg is partially supported by grant DMS2331634 from the National Science Foundation (NSF).


\begin{thebibliography}{99}

\bibitem{acf19}
A. Abiad, G. Coutinho, M. A. Fiol.
On the $k$-independence number of graphs.
\emph{Discrete Math.}, 342: 2875--2885, 2019.

\bibitem{acfnz21}
A. Abiad, G. Coutinho, M.A. Fiol, B.D. Nogueira, S. Zeijlemaker. Optimization of eigenvalue bounds for the independence and chromatic number of graph powers. \emph{Discrete Math.}, 345(3): 112706, 2022.

\bibitem{AZ2024} A. Abiad, S. Zeijlemaker. A unified framework for the Expander Mixing Lemma for irregular graphs and its applications. \emph{Linear Algebra and its Applications}, 702:19--45, 2024. 

\bibitem{AIM} AIM Minimum Rank – Special Graphs Work Group (F.~Barioli, W.~Barrett, S.~Butler, S.~M.~Cioaba, D.~Cvetkovi\'c, S.~M.~Fallat, C.~D.~Godsil, W.~H.~Haemers, L.~Hogben,
R.~Mikkelson, S.~Narayan, O.~Pryporova, I.~Sciriha, W.~So, D.~Stevanovi\'c, H. van der~Holst, K.~Vander Meulen, A.~Wangsness Wehe).
\newblock Zero forcing sets and the minimum rank of graphs.
\newblock {\em Linear Algebra and its Applications}, 428:1628–1648, 2008.
 
\bibitem{Alinathesis} F. Alinaghipour Taklimi. Zero Forcing Sets for Graphs. PhD thesis, University of Regina, 2013.

\bibitem{AB1995} T. Andreae, H.-J. Bandelt. Performance guarantees for approximation algorithms depending on parametrized triangle inequalities. \emph{SIAM Journal on Discrete Mathematics}, 8(1):1--16, 1995.

\bibitem{A} N.~Alon.
\newblock A propagation process on {C}ayley graphs.
\newblock 2008.
\newblock \url{ https://www.cs.tau.ac.il//~nogaa/PDFS/pn.pdf}.

\bibitem{AL2007v1} N. Alon, E. Lubetzky. Graph powers, Delsarte, Hoffman, Ramsey, and Shannon. \emph{SIAM J. Discrete Math.}, 21(2):329–348, 2007.

\bibitem{AL2007v2} N. Alon, E. Lubetzky. Independent sets in tensor graph powers. \emph{J. Graph Theory}, 54(1):73--87, 2007.

\bibitem{Apowers} N. Alon. Graph powers, \url{https://web.math.princeton.edu/~nalon/PDFS/cap2.pdf}.

\bibitem{powersthesis} E. Lubetzky. 
Graph Powers and Related Extremal Problems. PhD thesis, Tel Aviv University, 2007.

\bibitem{AM2002} N. Alon, B. Mohar. The chromatic number of graph powers. \emph{Combinatorics, Probability and Computing}, 11:1--10, 2002.

\bibitem{ACDP2015} D. Amos, Y. Caro, R. Davila, R. Pepper. Upper bounds on the $k$-forcing number of a graph. \emph{Discrete Appl. Math.}, 181:1--10, 2015. 

\bibitem{barioli2010} F. Barioli, W. Barrett, S. M. Fallat, H. Tracy Hall , L. Hogben, B. Shader, P. van den Driessche, Hein van der Holst. Zero forcing parameters and minimum rank problems. \emph{Linear Algebra and its Applications}, 433(2):401--411, 2010.

\bibitem{Basavaraju2014RainbowProducts} M. Basavaraju, L. Sunil Chandran, D. Rajendraprasad, A. Ramaswamy. Rainbow connection number of graph power and graph products. \emph{Graphs Combin}., 30(6):1363--1382, 2014.

\bibitem{BC2000} M. A. Bender, C. Chekuri. Performance guarantees for the TSP with a parameterized triangle inequality. \emph{Information Processing Letter}, 73(1-2):17--21, 2000.

\bibitem{BRESAR2017}
B. Bre{\u{s}}ar, C. Bujt\'as, T.~Gologranc, S. Klav\u{z}ar, G. Ko\u{s}mrlj, B.~Patk\'os, Z.~Tuza, M.~Vizer. Grundy dominating sequences and zero forcing sets. {\em Discrete Optimization}, 26:66--77, 2017.

\bibitem{burgarth2009local}
D.~Burgarth, S.~Bose, C.~Bruder, V.~Giovannetti.
\newblock Local controllability of quantum networks.
\newblock {\em Physical Review A}, 79(6):060305, 2009.


\bibitem{BY2013} S. Butler, M. Young. Throttling zero forcing propagation speed on graphs. \emph{AUSTRALASIAN JOURNAL OF COMBINATORICS}, 57:65--71, 2013.

\bibitem{Dprivate} R. Davila. Private communication,  2025, \url{https://txgraffiti.streamlit.app/}. 

\bibitem{dean2011power}
N.~Dean, A.~Ilic, I.~Ramirez, J.~Shen, K.~Tian.
\newblock On the power dominating sets of hypercubes.
\newblock {\em 2011 14th IEEE International Conference on Computational Science and Engineering}, 488--491, 2011.

\bibitem{Das2013LaplacianGraph}
K. Ch. Das, J.-M. Guo. Laplacian eigenvalues of the second power of a graph. \emph{Discrete Math.}, 313(5):626--634, 2013.


\bibitem{DM2016} L. Deaett and S. A. Meyer. The minimum rank problem for circulants.
\emph{Linear Algebra Appl.}, 491:386-418, 2016.

\bibitem{AGKSMY2012} L. de Alba, J. Grout, I-J. Kim, S. Kirkland, J. McDonald, and A. Yielding. Minimum rank of powers of trees. \emph{Electron. Linear Algebra}, 23:151-163,
2012.

\bibitem{EHHLR2012} C.J. Edholm, L. Hogben, M. Huynh, J. LaGrange, D.D. Row. Vertex and edge spread of zero forcing number, maximum nullity, and minimum rank of a graph. \emph{Linear Algebra and its Applications}, 436(12):4352--4372, 2012.

\bibitem{G2014} M. Goff. Edge growth in graph squares, pages 593--612. De Gruyter, Berlin,
Boston, 2014. 

\bibitem{H2009} H. Hajiabolhassan. On colorings of graph powers. \emph{Discrete Mathematics}, 309(13):4299--4305, 2009.

\bibitem{Hav2009} F. Havet. Choosability of the square of planar subcubic graphs with large girth. \emph{Discrete Mathematics}, 309:3353--3563, 2009.

\bibitem{H2011} P. Hegarty. A Cauchy-Davenport type result for arbitrary regular graphs. \emph{Integers}, 11(2):227-235, 2011.

\bibitem{kalinowski2019zero} T. Kalinowski, N. Kam\u{c}ev, B. Sudakov. The zero forcing number of graphs. \emph{SIAM Journal on Discrete Mathematics}, 33(1):95--115, 2019.

\bibitem{KP2016} R. J. Kang, François Pirot. Coloring powers and girth. \emph{SIAM Journal on Discrete Mathematics}, 30(4):1938--1949, 2016.

\bibitem{KK1969} F. Kramer, H. Kramer. Ein Färbungsproblem der Knotenpunkte eines Graphen bezüglich der Distanz p. \emph{Rev. Roum. Math. Pures Appl.}, 14(2):1031--1038, 1969.

\bibitem{l92} N. Linial. Locality in distributed graph algorithms. \textit{SIAM J. Comput.}, 21:193--201, 1992.


\bibitem{JZ2017} R. Xing, B. Zhou. On the two largest distance eigenvalues of graph powers. \emph{Information Processing Letters}, 119:39--43, 2017.

\bibitem{ZWWJ2022} W. Zhang, J. Wang, W. Wang, S. Ji. On the Zero Forcing Number and Spectral Radius of Graphs. \emph{The Electronic Journal of Combinatorics}, 29(1):\#P1.33, 2022.

\end{thebibliography}
\end{document}